\documentclass[12pt]{article}

\usepackage[T1]{fontenc} \usepackage[english]{babel}
\usepackage{latexsym, amsmath, amssymb, amsthm}

\usepackage{pgf,tikz}
\usetikzlibrary{arrows}

\begin{document}

\newtheorem{theorem}{Theorem}
\newtheorem{lemma}[theorem]{Lemma}
\newtheorem{corollary}[theorem]{Corollary}
\newtheorem{question}[theorem]{Question}
\newtheorem{proposition}[theorem]{Proposition}
\theoremstyle{definition}
\newtheorem{definition}[theorem]{Definition}
\newtheorem{example}[theorem]{Example}

\newcommand{\R}{\mbox{${\bf R}$}}
\newcommand{\Z}{\mbox{${\bf Z}$}}
\newcommand{\Q}{\mbox{${\bf Q}$}}
\newcommand{\N}{\mbox{${\bf N}$}}
\newcommand{\elt}{\mbox{$v \in G$}}
\newcommand{\elg}{\mbox{$\phi \in \Gamma$}}
\newcommand{\gstab}{\mbox{$G_v$}}
\newcommand{\caa}{\mbox{$C_1(g)$}}
\newcommand{\cia}{\mbox{$C_i(g)$}}
\newcommand{\ima}{\mbox{${\rm Im}$}}
\newcommand{\de}{\mbox{${\rm deg}$}}
\newcommand{\dist}{\mbox{${\rm dist}$}}
\newcommand{\DL}{\mbox{${\rm\sf DL}$}}
\newcommand{\wre}{\mbox{${\rm Wr\,}$}}
\newcommand{\res}{\mbox{$\mid$}}
\newcommand{\nin}{\mbox{$\ \not\in\ $}}
\newcommand{\fix}{\mbox{${\rm fix}$}}
\newcommand{\pr}{\mbox{${\rm pr}$}}
\newcommand{\sym}{\mbox{${\rm Sym}\;$}}
\newcommand{\aut}{\mbox{${\rm Aut}$}}
\newcommand{\auttn}[0]{\mbox{${\rm Aut}\;T_n$}}
\newcommand{\autx}[0]{\mbox{${\rm Aut}\;X$}}
\newcommand{\autga}[0]{\mbox{${\rm Aut}(\Gamma)$}}
\newcommand{\autde}[0]{\mbox{${\rm Aut}(\Delta)$}}
\newcommand{\out}{\mbox{${\rm out}$}}
\newcommand{\inn}{\mbox{${\rm in}$}}
\newcommand{\Al}{\mbox{${\rm Al}$}}
\newcommand{\Pl}{\mbox{${\rm Pl}$}}

\newcommand{\vZZ}{\vec{\mathbb{Z}}}

\renewcommand{\div}{{\hbox { \rm div }}}
\renewcommand{\mod}{{\hbox { \rm mod }}}
\newcommand{\id}{\hbox{\rm id}}
\newcommand{\ZZ}{\mathbb{Z}}
\newcommand{\Aut}{\mathrm{Aut}}
\newcommand{\tG}{{\widetilde{G}}}
\newcommand{\tH}{{\widetilde{H}}}
\newcommand{\tHpsi}{{\widetilde{H}}_\psi}
\newcommand{\CDC}{\mathop{{\rm CDC}}}
\newcommand{\CDHC}{\mathop{{\rm CDHC}}}
\newcommand{\g}{\mathbf g}
\newcommand{\h}{\mathbf h}
\newcommand{\V}{\mathrm V}
\newcommand{\VX}{{\mathrm V}X}
\newcommand{\E}{\mathrm E}
\newcommand{\A}{\mathrm A}
\newcommand{\TFaut}{{\mathop{{\rm TF}}}}
\newcommand{\Sym}{\mathrm{Sym}}
\newcommand{\Haut}{{\mathop{{\rm HAut}}}}
\newcommand{\Hautde}[0]{\mbox{${\rm HAut}(\Delta)$}}

\title{Infinite arc-transitive\\ and highly-arc-transitive digraphs}

\author{
R\"ognvaldur G.\ M\"oller\thanks{Science Institute, University of Iceland, IS-107 Reykjav\'ik, Iceland.  E-mail: {\tt roggi@hi.is} \newline R\"ognvaldur G.\ M\"oller acknowledges support from the University of Iceland Research Fund.}
 ,\quad
Primo\v{z} Poto\v{c}nik\thanks{Faculty of Mathematics and Physics, University of Ljubljana, Jadranska 19, SI-1000 Ljubljana, Slovenia. E-mail: {\tt primoz.potocnik@fmf.uni-lj.si} \newline
Primo\v{z} Poto\v{c}nik acknowledges the financial support from the Slovenian Research Agency (research core funding No. P1-0294).} ,\quad
Norbert Seifter\thanks{Montanuniversit\"at~Leoben, Franz-Josef-Strasse~18, A-8700 Leoben,~Austria.
E-mail: {\tt seifter@unileoben.ac.at}}
}

\maketitle

\begin{abstract}
\noindent
A detailed description of the structure of two-ended arc-transitive digraphs is given.
It is also shown that several sets of conditions, involving such concepts as Property Z, local quasi-primitivity and prime out-valency, imply that an arc-transitive digraph must be highly-arc-transitive.   These are then applied to give a complete classification of two-ended highly-arc-transitive digraphs with prime in- and out-valencies.
\end{abstract}

\section*{Introduction}  

A digraph $\Gamma$ is arc-transitive if the automorphism group acts transitively on the arcs.  If the automorphism group acts transitively on the set of $k$-arcs (directed paths of length $k$) then we say that $\Gamma$ is $k$-arc-transitive, and a digraph that is $k$-arc transitive for all $k$ is said to be highly-arc-transitive.  The class of highly-arc-transitive digraphs was introduced by Cameron, Praeger and Wormald in \cite{CPW1993}.  (Notation and terminology is explained in the next section.)  Various connections to other fields of research and the  challenging questions and conjectures in \cite{CPW1993} have led to many papers discussing this concept. Highly-arc-transitive digraphs exhibit a rich and varied structure and some kind of a general classification seems very unlikely.  

In this paper we study two-ended highly-arc-transitive digraphs with prime valency and get a number of strong and surprising results.  One such result is Corollary~\ref{CAlternets} which says that the alternets (called reachability digraphs in \cite{CPW1993}) in such a digraph are  complete bipartite digraphs.  This is a special case of Conjecture 3.3 in \cite{CPW1993}, the general case was disproved recently by DeVos, Mohar and \v{S}\'{a}mal in \cite{DeVosMoharSamal2015} and Neumann in \cite{Neumann2013}.  We also get a complete classification of two-ended highly-arc-transitive digraphs with prime in- and out-valencies (Corollary~\ref{CPartialLine}).  These graphs bear a close resemblance to the finite digraphs defined by Praeger and Xu in \cite{PraegerXu1989} and one can also see a kinship with \cite[Theorem~1.1]{PotocnikVerret2010}.  The final result (Theorem~\ref{TDescendants}) is largely unrelated to the above except that it further demonstrates the special properties of highly-arc-transitive digraphs with prime in- and out-valencies.   Further evidence of this phenomena can be found when one considers Question 1.2 in \cite{CPW1993} about whether or not the reachability relation is universal.    DeVos, Mohar and \v{S}\'{a}mal in \cite{DeVosMoharSamal2015} give examples of highly-arc-transitive digraphs where the reachability relation is universal, but if one assumes that the in- and out-valencies are both equal to some prime $p$ then the reachability relation cannot be universal, see Malni\v{c} et al.~\cite[Theorem~3.3]{Malnicetal2005}.  

The groundwork for the above results is done in Sections 1 and 2.  Some of the results have a highly technical flavour, but we do believe them to be of individual interest.
We provide several sets of conditions, involving such concepts as Property Z, local quasi-primitivity and prime out-valency, under which an infinite arc-transitive digraph must be highly-arc-transitive (see Theorem~\ref{TMain}, Corollary~\ref{CpMain} and Proposition~\ref{PValency2}).  We then give a detailed description of the structure of two-ended vertex- and arc-transitive digraphs (see Lemma~\ref{lem:ToZ} and Corollary~\ref{CToZ}).    The results for digraphs with prime in- and out-valencies then follow in Sections 3 and 4.

\section*{Notation and terminology}

A {\em digraph} $\Gamma$ is a pair of sets $(\V\Gamma, \E\Gamma)$ where $\V\Gamma$ is a set whose elements we call {\em vertices} and $\E\Gamma$ is a subset of $\V\Gamma\times \V\Gamma$ whose elements we call {\em arcs}.
If $(u,v)$ is an arc of $\Gamma$, then $u$ is called the {\em initial vertex} and $v$ the {\em terminal vertex} of the arc $(u,v)$.
We also allow, if $v\neq w$, that both $(v,w)$ and $(w,v)$ are arcs.  A (undirected) graph is a pair of sets $(\V\Gamma, \E\Gamma)$ where $\V\Gamma$ is a set whose elements we call {\em vertices} and $\E\Gamma$ is a set of 1 and 2 element subsets of $\V\Gamma$ whose elements we call {\em edges}.  Note that undirected graphs are also allowed to contain loops.  

For a vertex $v$ in a digraph $\Gamma$ 
we define the sets of {\em in-} and {\em out-neighbours} as
$$\inn(v)=\{u\in \V\Gamma\mid (u,v)\in \E\Gamma\}\qquad\mbox{and}\qquad
\out(v)=\{u\in \V\Gamma\mid (v,u)\in \E\Gamma\}.$$
A vertex that is an in- or out-neighbour of $v$ is said to be a {\em neighbour}.  
The cardinality of $\inn(v)$ is the {\em in-valency} of $v$ and the cardinality of $\out(v)$ is the {\em out-valency} of $v$.  A vertex is called a {\em source} if $\inn(v)=\emptyset$ and a {\em sink} if $\out(v)=\emptyset$.  A digraph where both the in- and out-valencies are finite is said to be {\em locally finite}.

An {\em automorphism} of a digraph $\Gamma$ is a bijective map $\varphi: \V\Gamma\to \V\Gamma$ such that $(u,v)$ is an arc in $\Gamma$ if and only if $(\varphi(u), \varphi(v))$ is an arc in $\Gamma$.  The group of all automorphisms is denoted with $\autga$.  For $g\in \autga$ we will write $v^g$ to denote the image of a vertex $v$ under $g$ and similarly for the images of arcs, $k$-arcs and subsets.  A {\em digraph homomorphism} from a digraph $\Gamma_1$ to a digraph $\Gamma_2$ is a map $\varphi: \V\Gamma_1\to \V\Gamma_2$ such that if $(u,v)$ is an arc in $\Gamma_1$ then $(\varphi(u), \varphi(v))$ is an arc in $\Gamma_2$.

A $k$-arc in a digraph
$\Gamma$ is a sequence of vertices $v_0, v_1, \ldots, v_k$ such that $(v_i, v_{i+1})$ is an arc for $i=0, 1, \ldots, k-1$.
Similarly, a sequence $v_0, v_1, v_2, \ldots$ such that 
$(v_i, v_{i+1})$ is an arc for $i=0, 1, \ldots$ is called
an {\em infinite arc} and a sequence $\ldots, v_{-2}, v_{-1}$, $v_0, v_1, v_2, \ldots$ such that 
$(v_i, v_{i+1})$ is an arc for every integer $i$ is called
a {\em two-way infinite arc}. 
Two-way infinite arcs  are often called {\em directed lines}.

A digraph $\Gamma$ is {\em vertex-transitive} if the automorphism group acts transitively on the vertex set and {\em arc-transitive} if the automorphism group acts transitively on the arc set.    If the automorphism group of $\Gamma$ acts transitively on the set of $k$-arcs in $\Gamma$ then we say that $\Gamma$ is $k$-arc-transitive.  When the automorphism group is $k$-arc-transitive for all $k$,  we say that $\Gamma$ is {\em highly-arc-transitive}.  Note that if $\Gamma$ is a locally finite highly-arc-transitive digraph then $\autga$ acts transitively on the set of two-way infinite arcs, see \cite[Lemma 1]{Moller2002}.

A {\em walk} $W$ in $\Gamma$ is a sequence $v_0, a_1, v_1, a_2, \ldots, a_n, v_n$ where the $v_i$'s are vertices and the $a_i$'s are arcs such that $a_i=(v_{i-1}, v_{i})$ or $a_i=(v_i, v_{i-1})$  for $i=0, 1, \ldots, n-1$.  The number $n$ is the {\em length} of the walk and will be denoted by $|W|$.  When $a_i=(v_{i-1}, v_{i})$ we say that the arc $a_i$ is {\em positively oriented}
in $W$ but {\em negatively oriented} otherwise.  A $k$-arc
is just a walk of length $k$ where all the arcs are positively oriented.
  If the vertices in a walk are all distinct then we speak about a {\em path}.  One-way infinite paths $v_0, a_1, v_1, a_2, v_3, \ldots$ are often called {\em rays}.   When the digraph is without loops and asymmetric, i.e.~there is no pair of vertices $u, v$ such that both $(u,v)$ and $(v,u)$ are arcs, then one can leave out the arcs when discussing walks and paths and just list the vertices.   A digraph is said to be {\em connected} if for any pair $v, w$ of vertices there is a path with initial vertex $v$ and terminal vertex $w$.  

We say that two arcs are {\em related} if they have a common initial or terminal vertex.  Let $R$ be the transitive closure of this relation.  The relation $R$ is clearly an 
$\Aut(\Gamma)$-invariant equivalence relation.  
 A  subdigraph spanned by one of the equivalence classes is called an {\em alternet}.  If the automorphism group is transitive on arcs then all the alternets are isomorphic.  In \cite{CPW1993} this relation (defined in a different way) is called the {\em reachability relation} and the alternets are called {\em reachability digraphs}.

One can regard the set of integers $\ZZ$ as an undirected graph in an obvious way, but we can also regard the integers as a digraph $\vZZ$ with 
arc-set $\{(i,i+1) : i \in \ZZ\}$.
  The {\em complete bipartite digraph} $\vec{K}_{r,s}$ is constructed from the ordinary complete bipartite graph $K_{r,s}$ by directing all edges from the side with $r$ vertices to the side with $s$ vertices.

Consider now a group $G$ acting on a set $\Omega$.  
Denote the image of a point $x\in \Omega$ under an element $g\in G$ by $x^g$.  The {\em stabilizer} of $a\in \Omega$ is the subgroup $G_a=\{g\in G\mid a^g=a\}$.   For a set $A\subseteq \Omega$ we denote the {\em pointwise stabilizer} of $A$ by
$G_{(A)}=\{g\in G\mid a^g=a\mbox{ for all }a\in A\}$ and the {\em setwise stabilizer} by $G_{\{A\}}=\{g\in G\mid A^g=A\}$.  A proper subset $B$ of $\Omega$ with at least two elements is called a {\em block of imprimitivity} (with respect to the action of $G$) if for every element $g\in G$ either $B^g=B$ or $B\cap B^g=\emptyset$.  If there are no blocks of imprimitivity then we say that the action is {\em primitive}. 
An equivalence relation on $\Omega$ that is preserved by $G$ is called a {\em $G$-congruence}.   If $B$ is a block of imprimitivity and $G$ is transitive then $B$ and  its translates, $B^g$ with $g\in G$, give us the classes of a $G$-congruence.  Conversely, if we have a non-trivial proper $G$-congruence then each one of its classes is a block of imprimitivity.
It is well-known that if $N$ is a normal subgroup of $G$ then the orbits of $N$ form the classes of a $G$-congruence on $\Omega$, i.e.~if $N$ acts non-trivially and intransitively then each orbit is a block of imprimitivity.  Hence, if $G$ acts primitively on $\Omega$ then a normal subgroup either acts trivially or transitively on $\Omega$.

When $\sigma$ is an equivalence relation on the vertex set of a digraph $\Gamma$ we can form the {\em quotient digraph}  $\Gamma/\sigma$ which has the set of $\sigma$-classes as a vertex set and if $A$ and $B$ are two $\sigma$-classes then $(A,B)$ is an arc in $\Gamma/\sigma$ if and only if there is a vertex $v\in A$ and a vertex $w\in B$ such that $(v,w)$ is an arc in $\Gamma$.  If $H$ is a subgroup of $\autga$ then $\Gamma/H$ denotes the quotient digraph of $\Gamma$ with respect to the equivalence relation whose classes are the $H$-orbits on the vertex set.  If $G\leq \autga$ and $\sigma$ is a $G$-congruence then $G$ has a natural action on the $\sigma$-classes that gives an action on the quotient digraph $\Gamma/\sigma$ by automorphisms.

\section{Digraphs with Property Z}

\begin{definition}
A connected digraph $\Gamma$ is said to have {\em Property Z} if there is a surjective digraph homomorphism $\varphi:  \Gamma\to \vZZ$.
\end{definition}

The following notation will be used in this section and the next two.  Let $\Gamma$ be a connected digraph with Property Z and $G=\autga$.  Suppose $\varphi:\Gamma\to \vZZ$ is a surjective digraph homomorphism.  
 Set $\Gamma_i=\varphi^{-1}(i)$.  Note that if $\psi: \Gamma\to \vZZ$ is another surjective digraph homomorphism then there is a number $k$ such that $\psi(v)=\varphi(v)+k$ for all vertices $v$ in $\Gamma$.  If $v$ is a vertex in $\Gamma$ then the digraph homomorphism $\varphi$ is completely determined by the value of $\varphi(v)$, and hence the collection of sets $\Gamma_i$ does not depend on the choice of the homomorphism $\varphi$.  

  Note that the $\Gamma_i$'s are the classes of a $G$-congruence on the vertex set of $\Gamma$.  If one of these classes is invariant under an element $g\in G$ then all of them are invariant under $g$.  Let $N$ denote the normal subgroup of $G$ that leaves 
   all the sets $\Gamma_i$ invariant.  We see that if $G$ is vertex-transitive then $N$ acts transitively on each of the $\Gamma_i$'s and the $\Gamma_i$'s are just the orbits of $N$.  We also see that if $g$ is an automorphism of $\Gamma$ such that $g$ takes some vertex in $\Gamma_0$ to a vertex in $\Gamma_1$ (by vertex transitivity such automorphisms exist) then $\Gamma_0^g=\Gamma_1$ and more generally $\Gamma_i^g=\Gamma_{i+1}$ and $g^{-1}G_{(\Gamma_i)}g=G_{(\Gamma_{i+1})}$ for all $i$.  

Define  $\Gamma_i^+$ as the subdigraph spanned by the set $\bigcup_{j\geq i} \Gamma_j$ and $\Gamma_i^-$ as the subdigraph spanned by the set $\bigcup_{j\leq i} \Gamma_j$.  Note that $\Gamma\setminus \Gamma_i$ is the disjoint union of the digraphs $\Gamma_{i-1}^-$ and $\Gamma_{i+1}^+$.  Since $G$ is the full automorphism group of $\Gamma$ we see that $G_{(\Gamma_i)}$ acts independently on $\Gamma_{i-1}^-$ and $\Gamma_{i+1}^+$, i.e. $G_{(\Gamma_i)}=G_{(\Gamma_{i}^-)}\times G_{(\Gamma_{i}^+)}$.  

\begin{lemma}\label{LBasciPropertyZ}
Let $\Gamma$ be a connected vertex-transitive digraph that has Property Z and $G=\autga$. 
\begin{description}
\item[(a)] If $G_{(\Gamma_0)}$ is non-trivial then $G$ is uncountable.
\item[(b)] If  $G_{(\Gamma_0)}$ is transitive on $\inn(v)$ or $\out(v)$ for some $v\in \Gamma_0$ then $\Gamma$ is highly-arc-transitive.
\end{description}
\end{lemma}

\begin{proof} (a)  From the assumptions we see that $G_{(\Gamma_0)}$ acts non-trivially on $\Gamma_1$ or $\Gamma_{-1}$.  We assume that $G_{(\Gamma_0)}$ acts non-trivially on $\Gamma_1$, the case where $G_{(\Gamma_0)}$ acts non-trivially on $\Gamma_{-1}$ is similar.

 Then, since $G_{(\Gamma_0)}=G_{(\Gamma_{0}^-)}\times G_{(\Gamma_{0}^+)}$, we see that $G_{(\Gamma_{0}^-)}$ acts non-trivially on $\Gamma_1$, and more generally, $G_{(\Gamma_{i-1}^-)}$ acts non-trivially on $\Gamma_i$.  For $i=1, 2, \ldots$ we choose an element $g_i\in G_{(\Gamma_{i-1}^-)}$ that acts non-trivially on $\Gamma_i$.  Any sequence $x=(x_i)_{i\geq 1}$ where each $x_i$ is either equal to $0$ or $1$ gives rise to an automorphism $g_x$ such that if $v$ is a vertex in $\Gamma_i$ then $v^{g_x}=v^{g_1^{x_1}g_2^{x_2} \cdots g_i^{x_i}}$. Clearly different sequences give rise to different automorphisms.   Hence $G$ is uncountable.
This part of the lemma and its proof is closely related to a well known result of Halin \cite[Theorem~6]{Halin1973}.

(b)  For the sake of the argument we assume that $G_{(\Gamma_0)}$ is transitive on $\out(v)$, the other case is similar.  Note that $G_{\{\Gamma_0\}}$ is transitive on $\Gamma_0$ and $G_{(\Gamma_0)}$ is normal in  $G_{\{\Gamma_0\}}$ and hence  $G_{(\Gamma_0)}$ acts transitively on $\out(w)$ for every vertex $w$ in $\Gamma_0$.  The observation that $G_{(\Gamma_0)}=G_{(\Gamma_{0}^-)}\times G_{(\Gamma_{0}^+)}$  implies that $G_{(\Gamma_0^-)}$ acts transitively on $\out(w)$ for every vertex $w$ in $\Gamma_0$.  If $w$ is a vertex in $\Gamma_i$ for some $i$ and $g\in G$ is such that $v^g=w$ then $G_{(\Gamma_i^-)}=g^{-i}G_{(\Gamma_0^-)}g^i$ acts transitively on $\out(w)$.  
From this we see that $\Gamma$ is arc-transitive.   We also see that if
 $\gamma=u_0, u_1, \ldots, u_k$ is a $k$-arc with $u_k\in \Gamma_i$ then $G_{(\Gamma_i^-)}$ acts transitively on $\out(u_k)$ and thus the pointwise stabilizer of the $k$-arc $\gamma$ acts transitively on $\out(u_k)$.  A standard induction argument now shows that $\Gamma$ is highly-arc-transitive.
\end{proof}

\begin{corollary}   Let $\Gamma$ be a connected vertex-transitive digraph that has Property Z and $G=\autga$. 
If $G_{(\Gamma_0)}$ is non-trivial and the sets $\Gamma_i$ are all finite then $G_{(\Gamma_0)}$ acts non-trivially on both $\Gamma_{-1}$ and $\Gamma_1$.
\end{corollary}

\begin{proof}
Let $m$ denote the number of elements in $\Gamma_i$.  Since $G_{(\Gamma_0)}$ is non-trivial it has to act non-trivially on $\Gamma_{-1}$ or $\Gamma_1$.  We will assume that that $G_{(\Gamma_0)}$ acts non-trivially on $\Gamma_1$.   By part {\bf (a)} of the preceding lemma the permutation group induced on $\Gamma_1^+$ is uncountable.  We can therefore find a number $n$ such that $G_{(\Gamma_0)}$ induces more than $m!$ distinct permutations on the set $\Gamma_1\cup\cdots\cup\Gamma_n$.  Thus there are two elements $g_1, g_2$ in $G_{(\Gamma_0)}$ that induce distinct permutations on the set $\Gamma_1\cup\cdots\cup\Gamma_n$ but induce the same permutation on the set $\Gamma_{n+1}$.  Then $g=g_1^{-1}g_2$ fixes all vertices in both $\Gamma_0$ and $\Gamma_{n+1}$ but does not fix some vertex in $\Gamma_1\cup\cdots\cup\Gamma_n$.  Hence there is some number $1< i\leq n+1$ such that $g$ fixes every vertex in $\Gamma_i$ but does not fix every vertex in $\Gamma_{i-1}$.  By vertex transitivity we can infer that  $G_{(\Gamma_0)}$ acts non-trivially on $\Gamma_{-1}$.
\end{proof}

The following Theorem is the basis for the results in Section~\ref{SLocally} on two-ended digraphs with prime in- and out-valencies.  We state the Theorem in a more general context than needed for those applications and then state the result for digraphs with prime in- and out-valencies as a Corollary.

A transitive group action is said to be {\em quasi-primitive}  if every normal subgroup either acts trivially or transitively.   Every primitive action is automatically also quasi-primitive and therefore any transitive group action on a set with a prime number of elements is quasi-primitive.  Furthermore any transitive action by a simple group is also quasi-primitive.
We say that a transitive digraph $\Gamma$ is {\em locally quasi-primitive} if the stabilizer of a vertex $v$ in the automorphism group acts quasi-primitively on both the sets $\inn(v)$ and $\out(v)$.   In the proof of the Theorem below local quasi-primitivity is precisely the condition needed to make the argument work.  This type of argument is common in permutation group theory and the theory of group actions on graphs and explains why the concept of quasi-primitivity occurs so frequently in the literature.

\begin{theorem}\label{TMain}
Let $\Gamma$ be a connected vertex-transitive digraph that has Property Z. Set $G=\autga$.  Assume that $G_v$ acts quasi-primitively on both $\inn(v)$ and $\out(v)$ for some vertex $v$ in $\Gamma$.  If $G_{(\Gamma_0)}$ is non-trivial then $\Gamma$ is highly-arc-transitive.  
\end{theorem}
 
\begin{proof} 
Let $v$ be a vertex in $\Gamma_0$.  Set $N=G_{(\Gamma_0)}$.  Suppose that $N$ acts trivially on both $\inn(v)$ and $\out(v)$.  Because $G_{\{\Gamma_0\}}$ is transitive on $\Gamma_0$ and $N$ is normal in $G_{\{\Gamma_0\}}$ we see that then $N$ would act trivially on both $\Gamma_{-1}$ and $\Gamma_1$.  If $g$ is an automorphism such that $\Gamma_0^g=\Gamma_1$ and then $G_{(\Gamma_1)}=g^{-1}Ng$ and thus $G_{(\Gamma_1)}$ acts trivially on $\Gamma_2$.     From this we conclude that $N$ acts trivially on $\Gamma_2$ since $N\leq G_{(\Gamma_1)}$.  In the same way we can show that $N$ acts trivially on $\Gamma_{-2}$.  By induction we can now show that if $N$ acts trivially on both $\inn(v)$ and $\out(v)$ then $N$ would act trivially on the whole of $\Gamma$. Because $\Gamma_0$ is invariant under $G_v$, we see that $N$ is a normal subgroup of $G_v$.  Since $G_v$ acts quasi-primitively on the set $\out(v)$ and $N$ acts non-trivially we see that $N$ acts transitively on $\out(v)$.  Now it follows from part (b) in Lemma~\ref{LBasciPropertyZ} that $\Gamma$ is highly-arc-transitive. \end{proof}

\begin{corollary}\label{CpMain}
Let $\Gamma$ be a connected arc-transitive digraph that has Property Z and $G=\autga$. Suppose that the in- and out-valencies of $\Gamma$ are both prime numbers.  If $G_{(\Gamma_0)}$ is non-trivial then $\Gamma$ is highly-arc-transitive.
\end{corollary}

\begin{proof} The assumptions that the in- and out-valencies are prime and that $\Gamma$ is arc-transitive imply that if $v$ is a vertex in $\Gamma$ then $G_v$ acts primitively, and thus also quasi-primitively, on both $\inn(v)$ and $\out(v)$. 
\end{proof}

For digraphs with in- and out-valencies both equal to 2 one gets a result analogous to the results of Tutte for trivalent graphs, see \cite{Tutte1947} and \cite{Tutte1959}.

\begin{proposition}\label{PValency2}
Let $\Gamma$ be a connected vertex-transitive digraph such that the in- and out-valencies are 2. Set $G=\autga$.  If $\Gamma$ is not highly-arc-transitive  then there is a number $k$ such that $G$ acts regularly on the $k$-arcs of $\Gamma$ and $|G_v|=2^k$ for every vertex $v$ in $\Gamma$.  In particular, if the automorphism group has infinite vertex stabilizers then $\Gamma$ is highly-arc-transitive.
\end{proposition}

\begin{proof}  Suppose that $\Gamma$ is not highly-arc-transitive.  Thus there is a number $k$ such that $\Gamma$ is $k$-arc-transitive but not $(k+1)$-arc-transitive.  If $\gamma=v_0, v_1, \ldots, v_k$ is some $k$-arc then the pointwise stabilizer of $\gamma$ must also fix the two vertices in $\out(v_k)$, otherwise $\Gamma$ would be $(k+1)$-arc transitive contrary to assumption.  By using the same method as Tutte used in the proof of the main theorem of \cite{Tutte1959} one can now show that then the pointwise stabilizer of an $k$-arc is trivial and thus $G$ acts regularly on the $k$-arcs in $\Gamma$.  The number of different $k$-arcs starting at a given vertex $v$ is $2^k$ and thus $|G_v|=2^k$.
\end{proof}

\section{Two-ended arc-transitive digraphs}

A digraph is said to have two ends if it is connected and one gets at most 2 infinite connected components by removing a finite set of vertices.   Formally the ends of a digraph (or a graph) are defined as equivalence classes of rays where two rays are said to be equivalent, or belong to the same end, if there is an infinite family of pairwise disjoint paths such that each one has its  initial vertex in one of the rays and the terminal vertex in the other.  This equivalence relation is clearly invariant under the action of the automorphism group and thus the automorphism group acts on the equivalence classes, i.e.~the ends.   The general structure of two-ended transitive graphs has been studied in several papers, see e.g.~\cite{ImrichSeifter1989}, \cite{JungWatkins1984} and \cite{DeVosMoharSamal2015}.   Note that by \cite[Theorem~7]{DiestelJungMoller1993} every two-ended vertex transitive graph is locally finite.

In the statement below, we let $\ZZ$ denote both the  group of integers
as well as the digraph with vertex-set $\ZZ$ and arcs of the form $(i,i+1)$ and $(i+1,i)$ for $i\in \ZZ$. Recall that
$\vZZ$ denotes the digraph with vertex-set $\ZZ$ and arcs of the form $(i,i+1)$ for $i\in \ZZ$. Moreover, we let
$D_\infty$ denote the infinite dihedral group. 

\begin{lemma}
\label{lem:ToZ}
Let $\Gamma$ be a vertex- and arc-transitive, two-ended digraph, 
and  $G=\autga$. Then $G$ contains
a normal subgroup $N$ having finite orbits on $\V\Gamma$, such that either:
\begin{itemize}
\item[{\rm (i)}] $\Gamma/N \cong \ZZ$ and $G/N \cong D_\infty$; or
\item[{\rm (ii)}] $\Gamma/N \cong \vZZ$ and $G/N \cong \ZZ$.
\end{itemize}
 In particular, $\Gamma$ has Property Z in case (ii) but not in case (i).
Moreover, if $\Gamma$ is $2$-arc-transitive, then {\rm (ii)} holds.

In case (i) $G$ acts transitively on the ends of $\Gamma$, but in case (ii) $G$ acts trivially on the ends of $\Gamma$.
\end{lemma}

\begin{proof}
Let $X$ be the underlying undirected graph of $\Gamma$ and $A=\Aut(X)$. Then $X$ is $2$-ended and
 $A$ acts transitively on the vertices of $X$. We may now use
 \cite[Proposition~3.2]{MollerSeifter1998} which says that $A$ contains a normal subgroup $K_o$
 acting with finite orbits on $\VX$ ($=\V\Gamma)$ such that $A/K_o$ is isomorphic either to $\ZZ$ or to $D_\infty$.  
Since $K_o$ is normal in $A$, the $K_o$-orbits on $\VX$ are the classes of an $A$-congruence $\sigma$ on $\VX$ and $A$ acts
transitively on the set of $\sigma$-classes. Let $K$ be the kernel of this action. Then $K$ has the same orbits on $\VX$ as $K_o$ and $A/K$ acts faithfully on the
vertices of the quotient graph $X/K$. Since $X/K$ is infinite and $K_o$ contains $K$, then $A/K$ is also infinite, and being 
a quotient of $A/K_o$, it is itself isomorphic to $\ZZ$ or $D_\infty$.

Let $N$ denote the kernel of the action of $G$ on $\Gamma/\sigma$.  Then $N$ has finite orbits on $\V\Gamma$ and $G/N$ acts faithfully and vertex-transitively, as well as arc-transitively,
on the quotient digraph $\Gamma/\sigma$. Hence $G/N$ embeds into $A/K$ and is isomorphic either to $\ZZ$ or $D_\infty$.  

 Now we consider the action of $G$ on the digraph $\Gamma/N$.   
Note that $G/N$ acts arc-transitively on $\Gamma/N$ so the stabilizer in $G/N$ of a vertex  in $\Gamma/N$
 acts transitively on the set of in-neighbours, as well as the set of out-neighbours, of that vertex.
 
If $G/N\cong \ZZ$, then (being an abelian group acting faithfully) $G/N$ acts regularly on the vertices of $\Gamma/N$, implying
that the in-valency and out-valency of $\Gamma/N$ is $1$. Hence, $\Gamma/N \cong \vZZ$.  

Suppose now that $G/N \cong D_\infty$. 
There are only two non-isomorphic faithful transi\-tive actions of $D_\infty$ on infinite sets: the regular action and the standard action of $D_\infty$ on $\ZZ$.
If the action of $G/N$ on $\Gamma/N$ is regular, then, as above, $\Gamma/N \cong \vZZ$, which is clearly a contradiction since 
$\Aut(\vZZ) \cong \ZZ$. Hence we may assume that $G/N$ acts upon $\Gamma/N$ as $D_\infty$ in its standard action on $\ZZ$.
It is now straightforward to show that $\Gamma/N\cong \ZZ$.  

Finally, the digraph $\ZZ$ is not $2$-arc-transitive since some $2$-arcs start and end in the same vertex and some do not. Hence $2$-arc-transitivity of $\Gamma$ implies (ii).  

The final statement about the action on the ends follows easily. 
\end{proof}  

\medskip\noindent
\begin{example}\label{Eladder}  In the proof above one can, with some extra work, show that each $\sigma$-class is an $N$-orbit and thus $\Gamma/N$ is equal to $\Gamma/\sigma$.  To do that one must use the assumption that the digraph is arc-transitive as the following example shows.

Define a digraph $\Gamma$ with vertex set $\ZZ\times\ZZ_2$ and arcs like shown in the Figure 1.

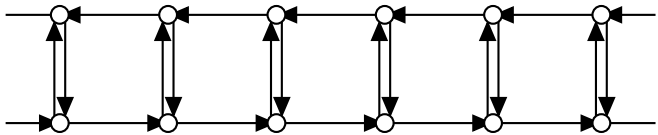
\begin{figure}[h]
\centering
\begin{tikzpicture}[scale=1.2,style=thick,x=1.2cm,y=1.2cm,>=triangle 45]
	\def\vr{2.8pt} 

	\draw[->] (-0.5,0)--(0,0);
	\draw[->] (0,0)--(1,0);
	\draw[->] (1,0)--(2,0);
	\draw[->] (2,0)--(3,0);
	\draw[->] (3,0)--(4,0);
    \draw[->] (4,0)--(5,0);
	\draw[-] (5,0)--(5.5,0);

	\draw[-] (0,1)--(-0.5,1);
	\draw[->] (1,1)--(0,1);
	\draw[->] (2,1)--(1,1);
	\draw[->] (3,1)--(2,1);
	\draw[->] (4,1)--(3,1);
    \draw[->] (5,1)--(4,1);
	\draw[->] (5.5,1)--(5,1);

%
	\draw[->] (-0.05,0.05)--(-0.05,.95);
	\draw[->] (0.05,0.95)--(0.05,0.05);
	\draw[->] (1-0.05,0.05)--(1-0.05,.95);
	\draw[->] (1+0.05,0.95)--(1+0.05,0.05);
	\draw[->] (2-0.05,0.05)--(2-0.05,.95);
	\draw[->] (2+0.05,0.95)--(2+0.05,0.05);
	\draw[->] (3-0.05,0.05)--(3-0.05,.95);
	\draw[->] (3+0.05,0.95)--(3+0.05,0.05);
	\draw[->] (4-0.05,0.05)--(4-0.05,.95);
	\draw[->] (4+0.05,0.95)--(4+0.05,0.05);
	\draw[->] (5-0.05,0.05)--(5-0.05,.95);
	\draw[->] (5+0.05,0.95)--(5+0.05,0.05);
%
    
	\draw (0,0) [fill=white] circle (\vr);
	\draw (1,0) [fill=white] circle (\vr);
	\draw (2,0) [fill=white] circle (\vr);
	\draw (3,0) [fill=white] circle (\vr);
	\draw (4,0) [fill=white] circle (\vr);
    \draw (5,0) [fill=white] circle (\vr);
	\draw (0,1) [fill=white] circle (\vr);
	\draw (1,1) [fill=white] circle (\vr);
	\draw (2,1) [fill=white] circle (\vr);
	\draw (3,1) [fill=white] circle (\vr);
	\draw (4,1) [fill=white] circle (\vr);
    \draw (5,1) [fill=white] circle (\vr);
\end{tikzpicture}
\caption{The graph $\Gamma$ in Example~\ref{Eladder}.}
\end{figure}

The automorphism group $G$ of this digraph is $D_\infty$ but the corresponding undirected graph $X$ is the infinite ladder graph and its automorphism group is $D_\infty\times \ZZ_2$.  The subgroup $K$ described in the above proof is equal to $\{e\}\times \ZZ_2$ and the $K$-orbits (that is, the $\sigma$-classes) are the sets $\{(i,0), (i,1)\}$.   But $N=G\cap K=\{e\}$.  Hence the $N$-orbits do not coincide with the $\sigma$-classes.  Note that $\Gamma$ is not arc-transitive.
\end{example}

If $\Gamma$ satisfies the assumptions of Lemma~\ref{lem:ToZ}, then  there exists a surjective digraph homomorphism
$\varphi$ from $\Gamma$ to either $\ZZ$ or $\vZZ$. Since the vertex-set of both these digraphs
is the set $\ZZ$, we may define $\Gamma_i=\varphi^{-1}(i)$ for $i\in\ZZ$.
The following Corollary lists several consequences of Lemma~\ref{lem:ToZ}. 

\begin{corollary}  \label{CToZ}
Let $\Gamma$ be a vertex- and arc-transitive two-ended digraph.
\begin{description}
\item[(a)] Case (i) in Lemma~\ref{lem:ToZ} occurs precisely when every vertex $v$ in $\Gamma_i$ has in- and out-neighbours both in $\Gamma_{i-1}$ and $\Gamma_{i+1}$.  The number of out-neighbours in $\Gamma_{i-1}$ is the same as the number of out-neighbours in  $\Gamma_{i+1}$ and the same holds for the number of in-neighbours.
The subdigraph spanned by $\Gamma_{i}\cup \Gamma_{i+1}$ is bipartite (i.e.~the underlying undirected graph is bipartite) and it is both vertex- and arc-transitive.
\item[(b)] Case (ii) in Lemma~\ref{lem:ToZ} occurs precisely when every vertex $v$ in $\Gamma_i$ has all of its in-neighbours in $\Gamma_{i-1}$ and all of its out-neighbours are 
in $\Gamma_{i+1}$.  The subdigraph spanned by $\Gamma_{i}\cup \Gamma_{i+1}$ is bipartite and it is  arc-transitive and all arcs have their initial vertex in $\Gamma
_i$ and terminal vertex in $\Gamma_{i+1}$.
\item[(c)] The in-valency of $\Gamma$ is equal to the out-valency.
\item[(d)] If the out-valency of $\Gamma$ is odd then $\Gamma$ has Property Z.
\item[(e)] If $\Gamma$ is 2-arc-transitive 
 then $\Gamma/N=\vZZ$ and $\Gamma$ has Property Z.
 \item[(f)]  If $\Gamma$ is locally quasi-primitive and the stabilizer of a vertex has more than 2 elements then $\Gamma$ has property Z.
  \item[(g)] Suppose that the out-valency is 2.  If the stabilizer of a 
vertex has more than 2 elements then $\Gamma$ has Property Z.
\end{description}
\end{corollary}

\begin{proof}  {\bf (a)} In this case we have in $\Gamma$  arcs with initial vertex in $\Gamma_i$ and terminal vertex in $\Gamma_{i+1}$, as well as arcs with  initial vertex in $\Gamma_{i+1}$ and terminal vertex in $\Gamma_{i}$.

If all the in- and out-neighbours of a vertex $u$ in $\Gamma_i$ belong to the set $\Gamma_{i-1}$, we say that $u$ is {\em backward-facing}.   If all the in- and out-neighbours of a vertex $v$ in $\Gamma_i$ belong to $\Gamma_{i+1}$ we say that $v$ is {\em forward-facing}.  First we look at the possibility that there exists a backward-facing vertex.  By vertex-transitivity every vertex is then either backward- or forward-facing.  All the neighbours of a backward-facing vertex are forward-facing and all the neighbours of a forward-facing vertex are backward-facing.  Now we have a contradiction, because in this setup $\Gamma$ would not be connected.

If all the in-neighbours of a vertex $u$ in $\Gamma_i$ belong to the set $\Gamma_{i+1}$ and all the out-neighbours of $u$ belong to the set $\Gamma_{i-1}$, we say that $u$ is {\em backward-oriented}.  Similarly, if  all the in-neighbours of  a vertex $v$ in $\Gamma_i$ belong to the set $\Gamma_{i-1}$ and all the out-neighbours of $v$ belong to the set $\Gamma_{i+1}$ we say that $v$ is {\em forward-oriented}.   Suppose now that there exists a forward-oriented or backward-oriented vertex.   Then every vertex is either backward- or forward-oriented.  All the neighbours of a backward-oriented vertex are also backward-oriented and all the neighbours of a forward-oriented vertex are also forward-oriented.  The assumptions  in case (i) of Lemma~\ref{lem:ToZ} imply that $\Gamma$ has both backward- and forward-oriented vertices.   Then these would be in different connected components of $\Gamma$ contradicting the assumption that $\Gamma$ is connected.

Now we see that every vertex in $\Gamma_i$ has in- and out-neighbours both in 
$\Gamma_{i-1}$ and $\Gamma_{i+1}$ and all the other claims in part {\bf (a)} follow.

Part {\bf (b)} is a direct consequence of vertex- and arc-transitivity and the assumption that $\Gamma/N\cong \vZZ$.
Parts {\bf (c)}, {\bf (d)} and {\bf (e)} follow easily from parts {\bf (a)} and {\bf (b)}.

{\bf (f)} Set $G=\autga$.  The subgroup $K$ fixing both ends of $\Gamma$ is a normal subgroup of $G$ and has either index 1 or 2 in $G$.  An automorphism $g$ in $K$ that fixes a vertex $v$ in $\Gamma_i$ cannot move a vertex in $\Gamma_{i-1}$ to a vertex in $\Gamma_{i+1}$ and vice versa.  Now $K_v=G_v\cap K$ is a normal subgroup of $G_v$ of index at most $2$.  The assumption that $|G_v|>2$ implies that this subgroup is non-trivial.  If $K_v$ were to act trivially on both the sets of in- and out-neighbours of $v$ then $K_v$ would act trivially on the whole of $\Gamma$ and would be trivial. Looking at the definition of quasi-primitivity we see that $K_v$ acts transitively on the set of in-neighbours or the set of out-neighbours of $v$.  In the case where $K_v$ acts transitively on the set of in-neighbours the vertex $v$ cannot have in-neighbours both in $\Gamma_{i-1}$ and $\Gamma_{i+1}$ and by part {\bf (a)} we are in case (ii) of Lemma~\ref{lem:ToZ} and $\Gamma$ has Property Z.  The case where $K_v$ is transitive on the set of out-neighbours is similar.  

{\bf (g)} Every digraph with both in- and out-valencies equal to 2 is locally quasi-primitive and the result now follows from part {\bf (f)}.
 \end{proof}

\begin{example}\label{Ezigzag}   In parts {\bf (f)} and {\bf (g)} of the above Corollary there is an \lq\lq additional\rq\rq\ assumption that the stabilizer of a vertex has more than 2 elements.  This assumption is needed as illustrated by the the following family $\{\Psi_n\}_{n\geq 2}$ of digraphs.  

Take two families $A_0, \ldots, A_{n-1}$ and $B_0, \ldots, B_{n-1}$ of directed lines.  Write $A_j=\ldots, a^{j}_{-1}, a^{j}_{0}, a^{j}_{1}, a^{j}_{2}, \ldots$ and $B_j=\ldots, b^{j}_{-1}, b^{j}_{0}, b^{j}_{1}, b^{j}_{2}, \ldots$.  For $i=\ldots, -1, 0, 1, 2, \ldots$ and $j=0, \ldots, n-1$ identify the vertex $b^{j}_i$ with the vertex $a^{j+i}_{-i}$ (addition performed modulo $n$) keeping all the arcs already present in the two families of lines.  The resulting digraph is two-ended and has both in- and out-valencies equal to 2.  Figure~2 shows a part of the digraph $\Psi_3$.   One can think of the horizontal directed lines as representing the directed lines $A_0, A_1, A_2$ and the \lq\lq zig-zag\rq\rq\ directed lines going from right to left as representing the directed lines $B_0, B_1, B_2$.

\begin{figure}[h]
\centering
\begin{tikzpicture}[scale=1.2,style=thick,x=1.2cm,y=1.2cm,>=triangle 45]
	\def\vr{2.8pt} 
	\draw[->] (-0.5,0)--(0,0);
	\draw[->] (0,0)--(1,0);
	\draw[->] (1,0)--(2,0);
	\draw[->] (2,0)--(3,0);
	\draw[->] (3,0)--(4,0);
    \draw[->] (4,0)--(5,0);
    \draw[-] (5,0)--(5.5,0);
	\draw[->] (-0.5,1)--(0,1);
	\draw[->] (0,1)--(1,1);
	\draw[->] (1,1)--(2,1);
	\draw[->] (2,1)--(3,1);
	\draw[->] (3,1)--(4,1);
    \draw[->] (4,1)--(5,1);
    \draw[-] (5,1)--(5.5,1);
	\draw[->] (-0.5,2)--(0,2);
	\draw[->] (0,2)--(1,2);
	\draw[->] (1,2)--(2,2);
	\draw[->] (2,2)--(3,2);
	\draw[->] (3,2)--(4,2);
    \draw[->] (4,2)--(5,2);
    \draw[-] (5,2)--(5.5,2);
%
    \draw[-] (0,0)--(-0.5,0.5);
    \draw[-] (0,1)--(-0.5,1.5);
    \draw[-] (0,2)--(-0.45,1.1);

    \draw[->] (1,0)--(0,1);
	\draw[->] (4,2)--(3,0);
	\draw[->] (3,1)--(2,2);
	\draw[->] (2,0)--(1,1);
	\draw[->] (1,2)--(0,0);
	\draw[->] (4,1)--(3,2);
	\draw[->] (3,0)--(2,1);
	\draw[->] (2,2)--(1,0);
	\draw[->] (1,1)--(0,2);
	\draw[->] (4,0)--(3,1);
	\draw[->] (3,2)--(2,0);
	\draw[->] (2,1)--(1,2);
    \draw[->] (5,0)--(4,1);
	\draw[->] (5,1)--(4,2);
	\draw[->] (5,2)--(4,0);
    \draw[->] (5.5, 1.5)--(5,2);   
    \draw[->] (5.5, 0.5)--(5,1);
    \draw[->] (5.45,0.9)--(5,0);
    
	\draw (0,0) [fill=white] circle (\vr);
	\draw (1,0) [fill=white] circle (\vr);
	\draw (2,0) [fill=white] circle (\vr);
	\draw (3,0) [fill=white] circle (\vr);
	\draw (4,0) [fill=white] circle (\vr);
    \draw (5,0) [fill=white] circle (\vr);
	\draw (0,1) [fill=white] circle (\vr);
	\draw (1,1) [fill=white] circle (\vr);
	\draw (2,1) [fill=white] circle (\vr);
	\draw (3,1) [fill=white] circle (\vr);
	\draw (4,1) [fill=white] circle (\vr);
    \draw (5,1) [fill=white] circle (\vr);
	\draw (0,2) [fill=white] circle (\vr);
	\draw (1,2) [fill=white] circle (\vr);
	\draw (2,2) [fill=white] circle (\vr);
	\draw (3,2) [fill=white] circle (\vr);
	\draw (4,2) [fill=white] circle (\vr);
    \draw (5,2) [fill=white] circle (\vr);
\end{tikzpicture}
\caption{The graph $\Psi_3$ described in Example~\ref{Ezigzag}.}
\end{figure}
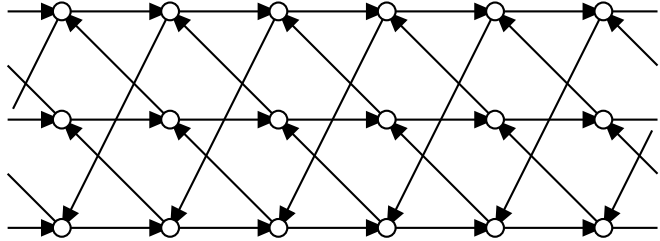

These graphs are vertex- and arc-transitive as can be seen easily by noticing that the roles of the $A_j$'s and the $B_j$'s can be interchanged.  These digraphs are locally quasi-primitive but do not have Property Z.

In part {\bf (f)} the assumption that the digraph is locally quasi-primitive is also needed as the following example demonstrates.  Suppose that $k$ is a positive integer.  We construct a new digraph from $\Psi_n$ as follows.  Replace each vertex $v$ in $\Psi_n$ with $k$ distinct vertices $v_1, \ldots, v_k$ and if $(v,w)$ is an arc in $\Psi_n$ then we let $(v_i, w_j)$ be an arc in our new digraph for all $i, j$ such that $1\leq i, j\leq k$.  This new digraph is also both vertex- and arc-transitive, it has two ends, the in- and out-valencies are $2k$ and the stabilizer of a vertex is infinite.  But this digraph is not locally quasi-primitive and it does not have property Z.
\end{example}

\section{Two-ended digraphs with prime in- and out-valencies}\label{SLocally}

In the setting of two-ended arc-transitive digraphs one can combine Theorem~\ref{TMain} with the description of two-ended arc-transitive digraphs.     In what follows we are always dealing with digraphs that have Property Z (see Corollary~\ref{CToZ}) and we adopt the notation defined in the previous section with the addition that $m$ denotes the number of elements in the set $\Gamma_i$ (all of these sets have the same number of elements).

\begin{theorem}\label{Ttwo-ended}
Let $\Gamma$ be a two-ended arc-transitive locally quasi-primitive digraph
 such that $m>2$ and $G=\autga$.  Let $v$ be a vertex in $\Gamma$ and suppose that $|G_v|>(m-1)!$.  Then $\Gamma$ is highly-arc-transitive.  In particular,  if the stabilizers of vertices are infinite then $\Gamma$ is highly-arc-transitive.
\end{theorem}

\begin{proof}  The assumption that $m>2$ together with the assumption that $|G_v|>(m-1)!$ imply that  $|G_v|>2$.  Hence, by Corollary~\ref{CToZ}{\bf (f)}, $\Gamma$ has Property Z.  If $v\in \Gamma_0$ then $G_{(\Gamma_0)}$ has index at most $(m-1)!$ in $G_v$.
The condition that $|G_v|>(m-1)!$ is thus enough to guarantee that $G_{(\Gamma_0)}$ is non-trivial and we can apply Theorem~\ref{TMain} to reach our conclusion. 
\end{proof}

We now move our attention to two-ended vertex- and arc-transitive digraphs of prime out-valency. Note that in this case the in- and out-valencies are equal, see Corollary~\ref{CToZ}{\bf (c)}.  Our analysis starts with a simple consequence of Theorem~\ref{Ttwo-ended} and culminates in a complete description of two-ended highly-arc-transitive digraphs of prime out-valency, see Corollary~\ref{CPartialLine}.

\begin{corollary}  \label{Ctwo-ended}
Let $\Gamma$ be a two-ended arc-transitive digraph of prime out-valency  $p$
 such that $m>2$ and $G=\autga$.  Let $v$ be a vertex in $\Gamma$ and suppose that $|G_v|>(m-1)!$.  Then $\Gamma$ is highly-arc-transitive.  In particular,  if the stabilizers of vertices are infinite then $\Gamma$ is highly-arc-transitive.
\end{corollary}

\begin{proof} Follows directly from Theorem~\ref{Ttwo-ended} above since an arc-transitive digraph with prime in- and out-valencies is locally quasi-primitive.   \end{proof}


The following two Corollaries can be rephrased so that the assumption of  prime out-valency is replaced with the assumption of local quasi-primitivity.

\begin{corollary}
Let $\Gamma$ be a two-ended arc-transitive digraph of prime out-valency  $p$
 such that $m>2$ and $G=\autga$.  Suppose that $G$ is $k$-arc-transitive where $k$ is a number such that $p^k>(m-1)!$.  Then  $\Gamma$ is highly-arc-transitive.  
\end{corollary}

\begin{proof} Given a vertex $v$ there are $p^k$ different $k$-arcs in $\Gamma$ that start with the vertex $v$.  Since $G$ acts transitively on the set of $k$-arcs we see that $|G_v|\geq p^k>(m-1)!$.  The conclusion now follows from the previous Corollary.  \end{proof}

Cameron, Praeger and Wormald conjectured \cite[Conjecture~3.3]{CPW1993} that the alternets in a two-ended highly-arc-transitive digraph are always complete bipartite digraphs.  Counterexamples were found by DeVos, Mohar and \v{S}\'{a}mal in \cite{DeVosMoharSamal2015} and Neumann in \cite{Neumann2013}.  The following Corollary shows that the conjecture is correct if one assumes in addition that the out-valency is a prime.

\begin{corollary}\label{CAlternets}
Let $\Gamma$ be a two-ended highly-arc-transitive digraph of prime out-valency  $p$
 such that $m>2$ and $G=\autga$.  Then the alternets of $\Gamma$ are
complete bipartite digraphs with $p$ vertices in each side.
\end{corollary}

\begin{proof} Let $v=v_1$ be a vertex in $\Gamma_0$ and $u=u_1$ a vertex in $\out(v)$.  Suppose $v_1, \ldots, v_p$ are the distinct in-neighbours of $u$.  All of the vertices $v_1, \ldots, v_p$ are in $\Gamma_0$.  The group $G_{(\Gamma_0)}$ fixes all the vertices $v_1, \ldots, v_p$ but acts transitively on $\out(v)$.  If the vertices in $\out(v_1)$ are $u_1, \ldots, u_p$ then clearly they all have the same in-neighbourhood.  Thus the subdigraph spanned by $\{v_1, \ldots, v_p, u_1, \ldots, u_p\}$ is the complete bipartite digraph $\vec{K}_{p,p}$ and it is clearly  an alternet of $\Gamma$.  \end{proof}

Before proceeding let us recap some ideas and results from \cite{PotocnikVerret2010}.  Note that in \cite{PotocnikVerret2010} the focus is on finite digraphs, but the results from \cite{PotocnikVerret2010} that we need (Lemma 3.1.(ii),  Lemma 3.2.(ii) and Lemma 3.3) are all proved without using the assumption of finiteness.
First we define the {\em digraph of alternets} of $\Gamma$, denoted with $\Al(\Gamma)$, as a digraph that has the set of alternets in $\Gamma$ as the vertex set and if $A$ and $B$ are alternets then $(A,B)$ is an arc if the intersection of the sinks of $A$ with the sources of $B$ is non-empty.
 
Following \cite{Malnicetal2008} we define the {\em reachability} relation $R_1^+$ such that two vertices $u$ and $v$ are related if there is a walk that starts at $u$ with a positively oriented arc and ends in $v$ with a negatively oriented arc and in between the arcs are alternatively positively and negatively oriented.   (In \cite{PotocnikVerret2010} this relation is denoted with ${\cal A}_1$.)
In the case where $\Gamma$ has no sinks and no degenerate alternets (an alternet is {\em degenerate} if it contains a 2-arc) then $\Al(\Gamma)$ is isomorphic with the quotient digraph one gets by contracting all 
$R_1^+$-classes in $\Gamma$, see \cite[Lemma 3.1.(ii)]{PotocnikVerret2010}. 

Next we define the {\em partial line graph} $\Pl(\Gamma)$ of a directed graph $\Gamma$.  The vertex set of $\Pl(\Gamma)$ is equal to the set of arcs in $\Gamma$ and two arcs $a_1 = (x, y)$ and $a_2 = (w, z)$ in $\Gamma$ form an arc $(a_1,a_2)$ in $\Pl(\Gamma)$ whenever $y = w$, and then $x, y, z$ is a 2-arc in $\Gamma$. (The partial line graph is more commonly called {\em the line digraph}.  Here we will use the term {\em partial line graph} since that is the term used in references \cite{MarusicNedela2001} and \cite{PotocnikVerret2010}.)  

The digraph $\Pl^r(\Gamma)$ is defined as the graph one gets by performing the  partial line graph operation successively $r$ times on $\Gamma$.  One could also define $\Pl^r(\Gamma)$ as the digraph one gets by considering the set of $r$-arcs in $\Gamma$ as the set of vertices and then say that if $a=(v_0, v_1,\ldots, v_r)$ and $b=(u_0, u_1,\ldots, u_r)$ are $r$-arcs in $\Gamma$ then $(a,b)$ is an arc in $\Pl^r(\Gamma)$ if $v_i=u_{i-1}$ for $i=1, 2, \ldots, r$.  The following Lemma is in essence proved in \cite[Proposition~2.1]{MarusicNedela2001}.

\begin{lemma}\label{LAutPl}
Let $\Gamma$ be a digraph having neither sinks nor sources.  Then $\aut(\Pl(\Gamma))\cong \autga$ and, more generally, $\aut(\Pl^r(\Gamma))\cong \autga$.
\end{lemma}

\begin{proof}
We only need to show that $\aut(\Pl(\Gamma))\cong \autga$  and then the latter statement follows by induction.  

Observe first that a mapping $\varphi$ which assigns to an arbitrary
element $g \in \autga$ the permutation $\varphi(g)$ of $\V\Gamma \times \V\Gamma$ defined by $(u,v)^{\varphi(g)} = (u^g,v^g)$ yields an embedding of $\autga$ into $\aut(\Pl(\Gamma))$. It remains to prove that
$\varphi$ maps $\autga$ onto $\aut(\Pl(\Gamma))$ surjectively.
In what follows we let $\iota(x)$ and $\tau(x)$ denote the initial and the terminal vertex of an arc $x$ of $\Gamma$, respectively.

Let $h\in \aut(\Pl(\Gamma))$ and $v\in \V\Gamma$. Since $\Gamma$ has no sources or sinks, we can define $w$ as an arbitrary in-neighbour of $v$ and  $u$ as an arbitrary out-neighbour of $v$.
Then $(w,v)$ is an in-neighbour of $(v,u)$ in $\Pl(\Gamma)$ and thus $(w,v)^h$ is an in-neighbour of $(v,u)^h$, implying that $\iota((v,u)^h) = \tau((w,v)^h)$. In particular, the vertex $\iota((v,u)^h)$ is independent of the choice of the out-neighour $u$ of $v$. This allows us to define 
a mapping $\psi(h) \colon \V\Gamma \to \V\Gamma$ by letting $v^{\psi(h)} = \iota((v,u)^h)$ where $u$ is an arbitrary out-neighbour of $v$, or equivalently by $v^{\psi(h)} = \tau((w,v)^h)$ where $w$ is an arbitrary in-neighbour of $v$.
 We see that $(v,u)^h = (\iota((v,u)^h), \tau((v,u)^h)) = (v^{\psi(h)},u^{\psi(h)})$. 

Now set $g=\psi(h)$ and $\bar{g} = \psi({h^{-1}})$ and observe that for an arbitrary 
arc $(v,u)$ of $\Gamma$ we have $(v,u) = (v,u)^{hh^{-1}} = (v^g,u^g)^{h^{-1}}$.
Since $(v^g,u^g) = (v,u)^h$ is an arc of $\Gamma$ the latter equals to $(v^{g\bar{g}},u^{g\bar{g}})$
and thus $v=v^{g\bar{g}}$.
Since $\Gamma$ has no sinks it thus follows that $g\bar{g}$ is the identity on $\V\Gamma$.
By changing the roles of $h$ and $h^{-1}$ we also see that $\bar{g}g$ is the identity on $\V\Gamma$,
implying that $g$ is a permutation on $\V\Gamma$.

Finally, recall that for an arbitrary arc $(u,v)$ of $\Gamma$ we have $(u^g,v^g) = (u,v)^h$ which is also an arc of $\Gamma$ showing that $g$ is an automorphism of $\Gamma$ such that $\varphi(g) = h$. In particular, 
$\varphi \colon \autga \to \Aut(\Pl(\Gamma))$ is surjective and hence an isomorphism between
$\aut(\Pl(\Gamma))$ and $\autga$.
\end{proof}

The alternets of a digraph $\Gamma$ are defined by considering a certain equivalence relation on the arc set of $\Gamma$ thus giving an equivalence relation $\alpha$ on the vertex set of the partial line graph $\Pl(\Gamma)$.  Two alternets are adjacent if they have a common vertex, say $v$, and then one of them contains an arc $e$ with $v$ as a terminal vertex and the other contains an arc $e'$ with $v$ as an initial vertex.  This says precisely that the equivalence classes of the vertices $e$ and $e'$ in $\Pl(\Gamma)$ are adjacent in the quotient graph $\Pl(\Gamma)/\alpha$.  On the other hand if two $\alpha$-classes are adjacent in $\Pl(\Gamma)/\alpha$ then one contains a vertex $e$ and the other a vertex $e'$ such that $(e,e')$ is an arc in $\Pl(\Gamma)$. This implies that in $\Gamma$ there is a 2-arc containing both the arcs $e$ and $e'$ and therefore the corresponding alternets intersect.   Thus $\Al(\Gamma)$ is isomorphic to $\Pl(\Gamma)/\alpha$.

If we assume that $\Gamma$ is asymmetric (i.e.~if $u$ and $v$ are distinct vertices then $(u,v)$ and $(v,u)$ cannot both be arcs in $\Gamma$) and has neither sinks nor sources then $\Al(\Pl(\Gamma))\cong \Gamma$, see \cite[Lemma 3.2.(ii)]{PotocnikVerret2010}.
Assuming in addition that $\Gamma$ is {\em loosely attached}  (two distinct alternets intersect in at most one vertex) and the alternets of $\Gamma$ are complete bipartite,
we can conclude that $\Pl(\Al(\Gamma))\cong \Gamma$, see \cite[Lemma 3.3]{PotocnikVerret2010}.

\medskip\noindent
{\bf Remark.}
 DeVos, Mohar and \v{S}\'{a}mal in \cite{DeVosMoharSamal2015} give examples of highly-arc-transitive digraphs with a degenerate alternet (i.e.~an universal reachability relation), thus answering Question 1.2 in \cite{CPW1993}.  But, if one assumes that the in- and out-valencies are both equal to some prime $p$ then there are no degenerate alternets, see Malni\v{c} et al.~\cite[Theorem~3.3]{Malnicetal2005}.  
\medskip

Define $\Delta_p$ as the digraph with vertex set $\ZZ\times \ZZ_p$ and arc set $((i, x), (i+1,y))$ for all $i\in\ZZ$ and $x,y\in\ZZ_p$.  The automorphism group of $\Delta_p$ is the unrestricted wreath product $\Sym(p)\,\wre\, \ZZ$, where $\Sym(p)$ is the full symmetric group on $\ZZ_p$.

\begin{corollary}\label{CPartialLine}
Let $\Gamma$ be a two-ended highly-arc-transitive digraph of prime out-valency  $p$
 such that $m>2$ and $G=\autga$.
Then $\Gamma$ is isomorphic to  $\Pl^r(\Delta_p)$ where $r$ is such that $p^{r+1}=m$.
\end{corollary}

\begin{proof}  From the assumptions it follows that $\Gamma$ is asymmetric, has neither sinks nor sources and contains no degenerate alternets.

  Consider a relation on the vertices in $\Gamma_0$ such that two vertices are related if they both belong to the set of sinks of the same alternet.  (Recall that by Corollary~\ref{CAlternets} the alternets are complete bipartite digraphs with $p$ vertices in each side.)   This is clearly an equivalence relation that is preserved by the group $G_{\{\Gamma_0\}}$.  The equivalence classes have size $p$.  We can also define a relation with the same properties by saying that two vertices are related if they are both sources of the same alternet.  Define the third equivalence relation by saying that vertices $u$ and $v$ in $\Gamma_0$ are related if they both belong to the set of sinks of the same alternet and they also belong to the set of sources of some common alternet.  This third equivalence relation is a refinement of the first two and is either trivial or each equivalence class has precisely $p$ elements.  In the latter case $\Gamma\cong \Delta_p$.  
  
  If the equivalence classes are just singletons then $\Gamma$ is loosely attached.  Then the digraph of alternets construction produces another digraph with two ends and since the digraph $\Gamma$ is loosely attached the digraph of alternets has also in- and out-valency $p$.    Note that if $\Gamma$ is highly-arc-transitive and the alternets are complete bipartite graphs then $\Al(\Gamma)$ is also highly-arc-transitive.  One can see this for instance by considering a $k$-arc in $\Al(\Gamma)$.  When the alternets are complete bipartite digraphs then this $k$-arc can be  \lq\lq lifted\rq\rq\ to a $(k+1)$-arc in $\Gamma$ and the result follows from the assumption that $\Gamma$ is highly-arc-transitive.  Clearly $\Al(\Gamma)$ is asymmetric, has neither sinks nor sources and contains no degenerate alternets.
  
If $m$ denotes the number of vertices in a fibre of a digraph homomorphism $\Gamma\to \vZZ$ then the number of vertices in a corresponding fibre in $\Al(\Gamma)$ is $m/p$.  Continuing in this way we will eventually reach the digraph $\Delta_p$ where there are precisely $p$ vertices in each fibre.  We can then get the digraph $\Gamma$ we started with by applying the partial line graph construction repeatedly to $\Delta_p$.  
\end{proof}

\begin{corollary}
Let $\Gamma$ be a two-ended highly-arc-transitive digraph of prime out-valency  $p$
 such that $m>2$ and $G=\autga$.  The group $\autga$ is isomorphic to  $\Sym(\ZZ_p)\,\wre\, \ZZ$.  In particular, if $p=2$ then the stabilizer of a vertex in $\autga$ is an elementary abelian 2-group.  
\end{corollary}

\begin{proof}
Follows directly from Lemma~\ref{LAutPl} and Corollary~\ref{CPartialLine}. Note that if $p=2$ then the stabilizer in $\autga$ of a vertex in $\Gamma$ is isomorphic to $(\ZZ_2)^\ZZ$.
\end{proof}

The {\em reverse} of the digraph $\Gamma$ is a digraph with the same vertex set as $\Gamma$ and $(v,u)$ is an arc in the reverse if and only if $(u,v)$ is an arc in $\Gamma$.    We say that $\Gamma$ is {\em skew-symmetric}  
 if $\Gamma$ and the reverse of $\Gamma$ are isomorphic.

There is no reason to expect a general highly-arc-transitive digraph to be skew-symmetric, for instance a regular directed tree with unequal in- and out-valencies is an example of a highly-arc-transitive digraph that is not skew-symmetric.  If one wants a two-ended highly-arc-transitive digraph that is not skew-symmetric then one can use Construction 2 from \cite{DeVosMoharSamal2015} as explained below.  Start by taking a regular finite edge-transitive but not vertex transitive undirected graph $T$, such graphs were first constructed by Folkman in \cite{Folkman1967}.  These graphs are necessarily bipartite
with the two parts of the bipartition, call them $A_1$ and $A_2$, being the orbits of the automorphism group.
Since we are assuming that the graph is regular then $|A_1|=|A_2|$.   Consider the graph as a digraph by orienting the edges so that all the arcs in the digraph have their initial vertex in $A_1$.  Construct a digraph $\Gamma$ with vertex set $\ZZ\times A_1\times A_2$ such that $((i, a_1, a_2), (j, b_1, b_2))$ is an arc if and only if $j=i+1$ and $(a_1, b_2)$ is an arc in $T$. It is shown in  \cite[Theorem~3.2]{DeVosMoharSamal2015} that $\Gamma$ is highly-arc-transitive.  Define $\Gamma_i$ as the set of all vertices of the form $(i, a_1, a_2)$ with $a_1\in A_1$ and $a_2\in A_2$.  The subdigraph of $\Gamma$ spanned by $\Gamma_i\cup \Gamma_{i+1}$ is isomorphic to the graph one gets by replacing each vertex in $T$ with $|A_1|=|A_2|$ vertices and each arc $(a_1, a_2)$ with a copy of the complete bipartite digraph $\vec{K}_{|A_1|,|A_2|}$.  This subdigraph is clearly not skew-symmetric and thus $\Gamma$ cannot be skew-symmetric.  

\begin{corollary}
Let $\Gamma$ be a connected two-ended highly-arc-transitive digraph with prime out-valency.  Then $\Gamma$ is skew-symmetric. 
\end{corollary}

\begin{proof}
Such graphs are characterized in Corollary~\ref{CPartialLine} and they are obviously skew-symmetric.
\end{proof}




\section{Descendant sets}

The {\em descendant set} of a vertex $v$ in a digraph $\Gamma$ is defined as the set of vertices $w$ such that for some $k$ there exists a $k$-arc with initial vertex $v$ and terminal vertex $w$.  In many known examples of highly-arc-transitive digraphs the subdigraph spanned by the descendant set of a vertex is a tree, but this is not always the case, see \cite[Example~2]{Moller2002}.  There are also examples of highly-arc-transitive digraphs where the subdigraph spanned by the set of descendants of a vertex is a tree but the digraph has only one end and is thus itself not \lq\lq tree-like\rq\rq, see \cite[Example~1]{Moller2002}.    

\begin{theorem}\label{TDescendants}
Let $\Gamma$ be an infinite  highly-arc-transitive connected digraph such that the in- and out-valencies are both equal to some prime $p$.  Then either $\Gamma$ is two-ended or the subdigraph spanned by the descendant set of a vertex $v$ is a rooted tree with out-valency $p$.
\end{theorem}

\begin{proof}  Let $L$ be some directed line in $\Gamma$.  Form the subdigraph $F$ spanned by the union of the descendant sets of all the vertices in $L$.  By \cite[Lemma~3]{Moller2002} the digraph $F$ has out-valency $p$, is highly-arc-transitive and has more than one end.  

If $F$ has just two ends then the in- and out-valencies of $F$ will be equal, see part {\bf (c)} in Corollary~\ref{CToZ}, and  we see that $F$ must be equal to $\Gamma$.  

In \cite[Theorem~3.4]{GrayMoller2011} it is shown that, if $\Gamma$ is a connected locally finite 2-arc-transitive digraph with more than two ends such that the stabilizer of a vertex $v$ acts primitively on the set $\out(v)$ then the digraph spanned by the set of descendants is a tree.  If the digraph $F$ has more than two ends then $F$ satisfies all these conditions, in particular, the stabilizer of a vertex acts transitively on $\out(v)$ and since the number of elements in $\out(v)$ is prime this action is primitive.  Hence the subdigraph in $F$ spanned by the descendants of a vertex $v$ is a tree, but this is the same as the subdigraph of $\Gamma$ spanned by the descendants of $v$.   
\end{proof}

\section*{Acknowledgements}

The authors want to thank Sara Zemlji\v{c} for help with producing the figures.


\bibliographystyle{abbrv}
\bibliography{references}

\end{document}